\documentclass[12pt]{amsart}
\usepackage[T1]{fontenc}
\usepackage[latin1]{inputenc}
\usepackage{amsmath}
\usepackage{amssymb}
\usepackage{graphics}
\usepackage{relsize}
\newtheorem{theo}{\bf Theorem}[section]
\newtheorem{propo}[theo]{\bf Proposition}

\newtheorem{defi}[theo]{\bf Definition}
\newtheorem{coro}[theo]{\bf Corollary}

\begin{document}

\title[Carlitz multipermutations]{Enumeration of Carlitz multipermutations}
\author{Henrik Eriksson}
\address{CSC \\
   KTH \\
   SE-100 44 Stockholm, Sweden}
\email{he@kth.se}
\author{Alexis Martin}
\address{MA \\ 
   EPFL\\
   CH-1015 Lausanne, Switzerland}
\email{alexis.martin@epfl.ch}
\date{2016-11-28}

\begin{abstract}
A multipermutation with $k$ copies each of $1\ldots n$ is Carlitz if neighbours
are different. We enumerate these objects for $k=2,3,4$ and derive recurrences.
In particular, we prove and improve a conjectured recurrence for $k=3$,
stated in OEIS, the Online Encyclopedia of Integer Sequences.
\end{abstract}
\maketitle

\section{Introduction}
\noindent
Leonard Carlitz \cite{Carlitz} enumerated compositions with
adjacent parts being different. We will count multipermutations
of $1^k,2^k,\ldots,n^k$ with the same condition. 
\begin{defi}
A multipermutation is {\em Carlitz} if adjacent elements are different.
\end{defi}

\noindent
For $k=1$, these
are just the $n!$ ordinary permutations, but for $k>1$ there are
few results. OEIS has entries A114938 for $k=2$, where an expression 
and a three-term recurrence is given, and A193638 for $k=3$,
but with no formula and only a conjectured recurrence. 

Let $A_k(n)$ be the set of Carlitz multipermutations of
$1^k,2^k,\ldots,n^k$ and let $a_k(n)=|A_k(n)|$. The simplest 
examples are 
\begin{flalign*}
A_2(2) & =\{1212,2121\},\ \ a_2(2)=2 &\\  
A_2(3) & =\{121323,123123,123132,123213,123231,\ldots\},\ \ a_2(3)=30
\end{flalign*}

\begin{table}[h]
\caption{Number of Carlitz mutipermutations}
\begin{tabular}{|r|r|r|r|r|r|r|r|}
\hline
$a_k(n)$ & $n=0$ & $n=1$ & $n=2$ & $n=3$ & $n=4$ & $n=5$ & $n=6$\\
\hline\hline
$k=1$ & 1 & 1 & 2 & 6 & 24 & 120 & 720\\
\hline
$k=2$ & 1 & 0 & 2 & 30 & 864 & 39 480 & 2 631 600\\
\hline
$k=3$ & 1 & 0 & 2 & 174 & 41 304 & 19 606 320 & 16 438 575 600\\
\hline
\end{tabular}
\end{table}
\noindent
The numbers grow very fast. An upper bound is of course 
$(k n)!/(k!)^n$, the number of all multipermutations.

We see that $A_2(2)$ has two elements, but only one {\em pattern}, $xyxy$.
If we identify elements with the same pattern, we get a smaller set 
$A_k^\prime(n)$. Every pattern may be realized in $n!$ ways as a 
multipermutation, so $a_k^\prime(n)=a_k(n)/n!$ as the examples show.
\begin{flalign*}
A_2^\prime(2) & =\{1212\},\ \ a_2^\prime(2)=1&\\   
A_2^\prime(3) & =\{121323,123123,123132,123213,123231\},\ \ a_2^\prime(3)=5
\end{flalign*}

\noindent
As representative we choose the {\em ordered multipermutation},
where the elements
appear in order. For any pattern, such as $zyzxyxyxz$, the order condition
determines what numeral each letter represents, in this case $121323231$.

Sometimes, it seems more natural to work with $a_k^\prime(n)$, sometimes
$a_k(n)$ is more convenient. OEIS has entries A278990 for $a_2^\prime(n)$, with 
formula and a three-term recurrence, and A190826 for  $a_3^\prime(n)$ with
no formula and an only conjectured recurrence.
\begin{table}[hbt]
\caption{Number of ordered Carlitz mutipermutations}
\begin{tabular}{|r|r|r|r|r|r|r|r|}
\hline
$a_k^\prime(n)$ & $n=0$ & $n=1$ & $n=2$ & $n=3$ & $n=4$ & $n=5$ & $n=6$\\
\hline\hline
$k=1$ & 1 & 1 & 1 & 1 & 1 & 1 & 1\\
\hline
$k=2$ & 1 & 0 & 1 & 5 & 36 & 329 & 3 655\\
\hline
$k=3$ & 1 & 0 & 1 & 29 & 1 721 & 163 386 & 22 831 355\\
\hline
\end{tabular}
\end{table}
\noindent

\section{Inclusion-exclusion formulas}
\noindent
Computing $a_2(n)$ by inclusion-exclusion is Example 2.2.3 in 
\cite{Stanley}. We show the method for $a_2(3)=30$. 

To begin with, there are $6!/2^3=90$ multipermutations of $1\,1\,2\,2\,3\,3$.
We subtract all containing the subpattern $1\!1$, i.e.~multipermutations of 
the five symbols $1\!1\,2\,2\,3\,3$. These are $5!/2^2$. The same goes for 
$2\!2$
and $3\!3$ so we subtract $\binom{3}{1}5!/2^2=90$. Patterns with both $1\!1$
and $2\!2$ were subtracted twice, so we add $4!/2^1$ for every such 
pair, totalling $\binom{3}{2}4!/2^1=36$. Finally, patterns with all three 
$1\!1$, $2\!2$, $3\!3$ must be subtracted, that is $\binom{3}{3}3!/2^0=6$.

The general formula looks like this.
\begin{propo}
\[
a_{2}(n)=\mathlarger{\sum}_{s+t=n}\Big[(-1)^{t}{n\choose s}
\frac{(2s+t)!}{(2!)^s}\Big]
\]
\end{propo}
\noindent
The sum is to be taken over nonnegative $s,t$ that add upp to $n$.
Here $s$ counts symbols that are separate, like $..x..x..$, and $t$ 
counts sybols that appear together, 
like $..xx..$, so there are $2s+t$ blocks to permute and $s$ 
indistinguishable pairs.

The case $k=3$ is trickier as we now have three subpatterns to consider.
If $s$ of the symbols appear separated, like $.x.x.x.$, $t$ of the symbols
appear two-plus-one, like $.xx.x.$, and $u$ of the symbols appear united,
like $.xxx.$, inclusion-exclusion will produce a surprisingly simple formula.
A more thorough treatment is given in Martin's thesis \cite{Martin}.
\begin{theo}
\[
a_{3}(n)=\mathlarger{\sum}_{s+t+u=n}\Big[(-1)^{t}{n\choose {s,t,u}}
\frac{(3s+2t+u)!}{(3!)^s}\Big]
\]
\end{theo}
\begin{proof}
A direct application of inclusion-exclusion would be possible if we knew
how many multipermutations contain $1\!1$, how many contain $1\!1$ and $2\!2$
etc. The $t=2, u=0$ counts permutations of blocks, some of length 1 and some
of length 2, for example $1\!1$ and $2\!2$. This will produce all desired
multipermutations, but some of them will be counted twice, for $1\!1 1$ is
the same sequence as $1 1\!1$. So we must subtract permutations where the
ones are united, and this explains the term $t=1, u=1$. But now again we
must add permutations with both $1\!1\!1$ and $2\!2\!2$ and this explains
the term $t=0, u=2$.  
\end{proof}

Let us try to compute $a_3(3)=174$ with the formula.   
$$1\cdot\frac{9!}{6^3}-3\cdot\frac{8!}{6^2}+3\cdot\frac{7!}{6^2}+3\cdot\frac{7!}{6}-
6\cdot\frac{6!}{6}+3\cdot\frac{5!}{6}-1\cdot\frac{6!}{1}+3\cdot\frac{5!}{1}-
3\cdot\frac{4!}{1}+1\cdot\frac{3!}{1}=174$$

\bigskip\noindent
It is easy to write down similar formulas for $k\ge 4$. We just give $k=4$
as an example. The proof has no new twists, so we omit it. Just note that 
$v$ and $w$ count $xx..xx$ resp.~$xxxx$.
\begin{theo}\label{fourformula}
\[
a_{4}(n)=\mathlarger{\sum}_{s+t+u+v+w=n}\Big[(-1)^{t+w}
{n\choose {s,t,u,v,w}}
\frac{(4s+3t+2u+2v+w)!}{(4!)^s(2!)^{v+t}}\Big]
\]
\end{theo}

\noindent
We were able to give each term a combinatorial interpretation but the formulas
are not new. Ira Gessel \cite{Gessel} used rook polynomials to derive more general 
expressions than these and Jair Taylor \cite{Taylor} proved the same
formulas directly from the generating function. Their elegant version
of Th.\ref{fourformula} is
$$
a_4(n)=\Phi((\frac{t^3}{6}-t^2+t)^n)\ ,
$$
where $\Phi(t^n)=n!$, so after expansion each power of $t$ is replaced
with a factorial. 

\section{Recurrences}
\noindent
For many purposes,
recurrences are superior to the explicit formulas of last section. We
will show how to get recurrences for $a_k^\prime(n)$, the number of ordered
Carlitz multipermutations. Recall that $a_k^\prime(n)=a_k(n)/n!$.

The OEIS \cite{OEIS} gives conjectural three-term recurrences for 
$a_2(n)$ and $a_2^\prime(n)$, a four-term recurrence for $a_3(n)$
and a five-term recurrence for $a_3^\prime(n)$. All these conjectures 
will be proved below.

\begin{theo}\label{tworec}
The sequence $p_n$, recursively defined by
$$p_{n+1}=(2n+1)p_{n}+p_{n-1}, \ p_0=1,\ p_1=0,$$
counts ordered Carlitz words of $1^2,\ldots,n^2$.
\end{theo}

\begin{proof}
As $p_n=a_2^\prime(n)$, $p_2=1$ counts the word {\tt 1212} and 
$p_3=5$ counts the words
{\tt 010212,012012,012102,012120,012021}, using symbols {\tt 012}. 
The first four words are of the type {\tt 0..\^{0}.}, that is the zero 
may be removed without violating the Carlitz property,
but the fifth word is of the type {\tt 0..x0x.}. 

\noindent
Now, we count words in $ 0^2,1^2,\ldots,n^2$. according to type.

\smallskip\noindent
{\tt 0..\^{0}.} is counted by $2n p_n$ (insert {\tt \^{0}} anywhere).

\smallskip\noindent
{\tt 0..x0x} for $x>1$ is counted by $p_n$ (transform {\tt 1..1 $\mapsto$  0..x0x}).

\smallskip\noindent
{\tt 0101..} is counted by $p_{n-1}$ (prefix {\tt 0101}).
\end{proof}
\noindent
In our example, {\tt 1212 $\mapsto$  02x0x2}, which is the same pattern as {\tt 012021}.

\medskip\noindent
\begin{theo}\label{threerec}

The sequences $p_n, q_n$, recursively defined by
$$
\begin{array}{rcll}
2p_{n+1}&=&(3n+3)q_n-2(3n+1)p_{n}+2p_{n-1},& p_0=1,\ p_1=0,\\
q_{n}&=&(3n+2)p_{n}+2q_{n-1},& q_0=0,
\end{array}
$$
count ordered Carlitz words of $1^3,\ldots,n^3$
resp.~of $0^2,1^3,\ldots,n^3$.
\end{theo}

\begin{proof}
$p_2=1$ counts the word {\tt 121212} and 
$q_2=8$ counts the words
{\tt 01\^{0}21212,\ldots,0121212\^{0},01202121,01212021}.
The first six words of the type {\tt 0..\^{0}.} are counted by $3n p_n$, 
the last two {\tt 0..x0x..x.} and {\tt 0..x..x0x.} with $x>1$ by
$2p_n$. Finally,
{\tt 0101..1.} and {\tt 01..101.} are counted by $2q_{n-1}$.
This proves the recurrence for $q_n$.

We now count $p_{n+1}$ by cases according to type of $\tt 0$. As there are
two noninitial zeros, the cases will sum to $2p_{n+1}$.\\
\smallskip\noindent
{\tt 0..0..\^0.} is counted by $(3n-1)q_n$ (insert {\tt\^0} in empty slot).
 
\smallskip\noindent
{\tt 0..0..x..x0x.} for $x>1$ is counted by $2(q_n-p_n-q_{n-1})$,
for our transformation {\tt 1..0..1..1. $\mapsto$  0..0..x..x0x.} does not work for 
{\tt 101..1.} (counted by $p_n$) or for {\tt 10..1..1.} (counted by $q_{n-1}$) .
  
\smallskip\noindent
{\tt 0101..1..0} and the equinumerous {\tt 01..101..0} split into subcases depending on the 
position of the other zero.\\\smallskip\indent
{\tt 010101..} is counted by $p_{n-1}$.\\
\smallskip\indent
{\tt 01010..1.}, {\tt 0101..01.} and {\tt 0101..10.} are counted by $3q_{n-1}$.\\
\smallskip\indent
{\tt 0101..0..1.}  and {\tt 0101..1..0.} are counted by $2p_n$.

\medskip\noindent
Collecting terms and replacing $2q_{n-1}$ with $q_n-(3n+2)p_n$
we get the recurrence for $p_{n+1}$.
\end{proof}

\noindent

\begin{coro}
The recursively defined sequence
$$
p_{n+1}=\lambda p_n+\mu p_{n-1}+\nu p_{n-2},\  p_0=1,\  p_1=0,\  p_2=1
$$
where $\lambda=(9n^2\!+\!9n\!+\!8)/2\ +2/n $, $\mu=(6n\!+3) -4/n $, $\nu=-2-2/n$
counts ordered Carlitz words of $1^3,\ldots,n^3$
\end{coro}
\begin{proof}
Lowering indices in Th.\ref{threerec} we get
$$2p_n=3nq_{n-1}-2(3n-2)p_{n-1}+ 2p_{n-2}
$$
Adding $-2-\frac{2}{n}$ times this to the $2p_{n+1}$-recurrence and then using the
$q_n$-recurrence , we get the desired four-term recurrence.
\end{proof}

\bigskip\noindent
The five-term recurrence in OEIS entry A190826 was found by Richard J.~Mathar
using an ansatz with twenty unknown coefficients \cite{Mathar}. It is of course
easily derived by adding two versions of our four-term recurrence, one of them 
with lowered indices. 

The four-term recurrence in OEIS entry A193638 was found by Alois P.~Heintz.
It is now a corollary obtained by multiplication with $(n+1)!$.

\noindent
Recurrences for $a_k^\prime(n)$ with $k>3$ may be derived in exactly the same way.
We state the result for $k=4$ here and leave the details to the reader.

\medskip\noindent
\begin{theo}\label{fourrec}

The sequences $p_n, q_n, r_n$, recursively defined by
$$
\begin{array}{rcl}
3p_{n+1}&=&(4n+1)q_n+3(10q_{n-1}-r_n+4r_{n-1}+(6n+7)p_n+p_{n-1})\\
2q_{n}&=&(4n+6)r_{n}+6r_{n-1}-(16n+6)p_{n}\\
r_{n}&=&(4n+3)p_{n}+3q_{n-1},\ p_0=1,\ p_1=0,\ q_0=0,\ r_0=0
\end{array}
$$
count ordered Carlitz words of $1^4,\ldots,n^4$ resp.~of
$0^3,1^4,\ldots,n^4$, and of
$0^2,1^4,\ldots,n^4$.
\end{theo}

\end{document}